\documentclass{cimart}

\usepackage{nicematrix}
\usepackage{tikz}

\title{The Pell tower and Ostronometry}

\author{Robbert Fokkink}

\authorinfo{Institute of Applied Mathematics, Delft University of Technology, Netherlands}{r.j.fokkink@tudelft.nl}

\abstract{%
    Conway and Ryba considered a table of bi-infinite Fibonacci sequences and discovered new interesting patterns. We extend their considerations to tables that are defined by the recurrence $X_{n+1}=dX_n+X_{n-1}$ for natural numbers $d$. In our search for new patterns we run into a Red Wall and exotic numeration systems.
    }



\VOLUME{33}
\YEAR{2025}
\ISSUE{2}
\NUMBER{6}
\DOI{https://doi.org/10.46298/cm.12261}

\begin{document}

\tableofcontents\bigskip\bigskip

The Fibonacci sequence $(F_n)=0,1,1,2,3,5,8,13,21,34,\ldots$ and its companion,
the Lucas sequence $(L_n)=2,1,3,4,7,11,18,29,47,\ldots$ 
have been examined inside out. 
A few years ago, John Conway and Alex Ryba managed
to find a surprising new shape in the Fibonacci recurrence $X_{n+1}=X_n+X_{n-1}$, which they called 
the \emph{Empire State Building}. We will explain what that is in Section \ref{section1}. In this paper we adopt their point of view 
and look for similar shapes in the 
recurrence $X_{n+1}=dX_{n}+X_{n-1}$ for general~$d$.
We call them Pell Towers, since the linear recurrence for $d=2$ produces the Pell numbers. The analysis of Conway and Ryba depends on relations between Fibonacci numbers which they called Fibonometry, since they can be derived from trigonometric relations. We extend this to the recursion 
$X_{n+1}=dX_n+X_{n-1}$ and call it Ostronometry.

\section{The Wythoff array and the Empire State building}\label{section1}

The starting point of Conway and Ryba in \cite{conwayryba} is Table~\ref{table:1} 
of Fibonacci recurrent sequences. 
\begin{table}[ht!]
\tiny{
\[
\begin{NiceArray}{cccccccccccc}
0&1&1&2&3&5&8&13&21&34& \cdots&  \\
1&3&4&7&11&18&29&47&76&123& \cdots  \\
2&4&6&10&16&26&42&68&110&178& \cdots \\
3&6&9&15&24&39&63&102&165&267& \cdots \\
4&8&12&20&32&52&84&136&220&356& \cdots \\
5&9&14&23&37&60&97&157&254&411& \cdots \\
6&11&17&28&45&73&118&191&309&500& \cdots \\
7&12&19&31&50&81&131&212&343&555& \cdots \\
8&	14&	22&	36&	58&	94&	152&	246&	398&	644& \cdots \\	
9&	16&	25&	41&	66&	107&	173&	280&	453&	733& \cdots \\
 \vdots & \vdots & \vdots & \vdots & \vdots & \vdots & \vdots & \vdots & \vdots & \vdots & \ddots 
 \CodeAfter
  \tikz \draw [line width=0.6 mm] (1-|1) -- (1-|3) |- (12-|3);
\end{NiceArray}
\]}
\caption{\small{The first ten rows of the Garden State. If we ignore the first two columns (called the
seed and the wall by Conway and Ryba), then
we get the Wythoff array in which
each natural number occurs exactly once.}}\label{table:1}
\end{table}

They call it the Garden State, since the
table grows like a garden of numbers from two initial columns.
In this section we review their results.

The table had been encountered before, without the initial two columns. 
Morrison ~\cite{morrison} found it first, modifying an earlier array of Stolarsky~\cite{stolarsky}.
He proved that each Fibonacci recurrent sequence with positive terms
occurs exactly once in the array (after deleting or adding
some initial terms to Stolarsky's array).
He also proved that the
rows consist of losing positions in Wythoff's game, which is why he called it
the Wythoff array. The study of this game and its associated numeration systems
is a world of its own~\cite{duchene}.

Kimberling~\cite{kimberling} found a neat connection between the Wythoff array and Zeckendorf numeration.
In this numeration system, each natural number $N$
is written as a sum of non-consecutive Fibonacci numbers
\begin{equation}\label{eq:Z}
N=\sum_{2\leq j\leq i} d_jF_j
\end{equation}
with $F_i$ the largest Fibonacci number bounded by $N$ and $d_j\in\{0,1\}$ and $d_{j-1}=0$ if $d_j=1$.
Each number can be represented by a binary word $d_2d_3\cdots d_i$, starting with the digit of $F_2$.
For instance, the binary word of $15$ is $010001$ in lsd representation.
This is the least significant digit first, or lsd representation of $N$.
Note that it runs in the opposite direction that we are all used to in our decimal
notation (the msd representation).
The expansion of Equation~\eqref{eq:Z}
can be found by repeatedly subtracting the largest possible Fibonacci number.
This process produces the digits in msd order.

If we arrange the words that start and end with a~$1$ in increasing length, and within constant length in increasing lexicographic order, then we get
\begin{equation}\label{eq:radix}
1, 101, 1001, 10001, 10101, 100001, 101001, 100101,\ldots.
\end{equation} 
This is known as the \emph{radix order}.
These words represent the initial column of the Wythoff array, next
to the wall in Table~\ref{table:1}.
The second column has prefix~$01$,
the third have prefix~$001$, and so forth.
To locate a number in the array, simply determine its Zeckendorf expansion.
The prefix of zeros determines the column. The suffix determines the row.

Conway and Ryba extended the Wythoff array by two columns, called
the \emph{seed} and the \emph{wall}. 
The seed $0,1,2,3,\ldots$ numbers the floors of the
Empire State Building, counting from the top. 
A seed term and a wall term determine the state and the recurrent sequence
grows from there.
Conway and Ryba introduced the operation $n\mapsto\mathtt{out}(n)$, which
prepends a zero. 
If $n$ is represented by $w$ then $\mathtt{out}(n)$ 
is represented by $0w$.
In the Wythoff array $\mathtt{out}$ takes a step to the right.
It has a hiccup for the extended array with the seed and the wall. 
If we apply $\mathtt{out}$ to the seed, then we get $0, 2, 3, 5, 7, 8, \ldots$
which is the wall minus one.
If we apply $\mathtt{out}$ to the wall,
then we get the sequence $2,5,7,10,13,15,\ldots$ which is the
first column of the Wythoff array plus one. From then
on, the operation runs smoothly.
\begin{table}[ht!]
\tiny{
\begin{center}
\[ 
\begin{NiceArray}{ccccccccccc}
\cdots &13&-8&5&-3&2&-1&1&0&1&\\
 \cdots &18&-11&7&-4&3&-1&2&1&3\\
 \cdots &10&-6&4&-2&2&0&2&2&4\\
 \cdots &15&-9&6&-3&3&0&3&3&6\\
 \cdots &20&-12&8&-4&4&0&4&4&8\\
 \cdots &12&-7&5&-2&3&1&4&5&9\\
 \cdots &17&-10&7&-3&4&1&5&6&11\\
 \cdots &9&-5&4&-1&3&2&5&7&12\\
 \cdots &14	&-8	&6	&-2	&4	&2	&6	&8	&14\\
 \cdots &19	&-11	&8	&-3	&5	&2	&7	&9	&16\\
 \iddots & \vdots & \vdots & \vdots & \vdots & \vdots & \vdots & \vdots & \vdots & \vdots 
  \CodeAfter 
\tikz \draw [line width=0.6 mm] (12-|11) |- (1-|11) -- (1-|8) |- (2-|6) |- (3-|4) |- (6-|2) |- (12-|2);
 \end{NiceArray}
 \]
\caption{\small{Extending the recurrence to the left. The wall reappears, with positive terms, and so does the seed in front of it, with negative terms. The order is shuffled. For instance, the seed 5 (sixth row) appears two rows down (eighth row) next to the left wall term as -5. It is a longstanding open problem whether there is an algorithm to decide if a bi-infinite linear recurrence sequence has only non-negative terms, see~\cite{kenison}.}}\label{table:2}
\end{center}}
\end{table}

The Fibonacci recurrence extends to negative indices.
Conway and Ryba consider the extension of Table~\ref{table:1}
to the left and call it the \emph{ExtraFib} array.
The numbers that appear here are alternating in sign.
If we take absolute values $Y=|X|$, 
then we retrieve the Fibonacci recurrence $Y_{-n-1}=Y_{-n}+Y_{-n+1}$,
which now grows to the left.  
Each sequence of the Wythoff array reappears on the left with alternating
signs, possibly at a different level.
If it reappears at the same level, then the bi-infinite
sequence is \emph{palindromic} (ignoring the signs).
Since sequences reappear at the left there is another wall (and another seed) on the left,
see Table~\ref{table:2}.
The left wall extends ever further to the left.

Conway and Ryba prove that 
the number of terms between the walls is odd, which is why it is natural
to center the table around the middle term (the pillar). 
The resulting figure in Table~\ref{table:3}
has the outline of a skyscraper, made up 
of blocks of constant width (number of terms between the walls)
which get longer and longer as we go down the table.
The palindromic sequences are evenly spaced within each block. 
By underlining these rows within the walls, the structure gets even more likeliness to the 
\emph{Empire State Building}.
Within a block, the palindromes are either
multiples of the Fibonacci numbers or of the Lucas numbers. That is why Conway and Ryba
call them Fifi blocks and Lulu blocks. 
They show in~\cite{conwayryba} that the Empire State Building has lots of structure.
What types of buildings do we get for other recursions?
We consider this question for the recursion $X_{n+1}=dX_n+X_{n-1}$.
\begin{table}[ht!]
\tiny{
\[ 
{
\begin{NiceArray}{rrrrrrr||r||rrrrrrr}[columns-width=auto]
\cdots &-8&5&-3&2&-1&1&0&1&1&2&3&5&8&\cdots\\
 \cdots &18&-11&7&-4&3&-1&2&1&3&4&7&11&18&\cdots\\
 \cdots &-16&10&-6&4&-2&2&0&2&2&4&6&10&16&\cdots\\
 \cdots &-24&15&-9&6&-3&3&0&3&3&6&9&15&24&\cdots\\
 \cdots &-32&20&-12&8&-4&4&0&4&4&8&12&20&32&\cdots\\
 \cdots &31&-19&12&-7&5&-2&3&1&4&5&9&14&23&\cdots\\
 \cdots &44&-27&17&-10&7&-3&4&1&5&6&11&17&28&\cdots\\
 \cdots &23&-14&9&-5&4&-1&3&2&5&7&12&19&31&\cdots\\
 \cdots &36&-22&14	&-8	&6	&-2	&4	&2	&6	&8	&14&22&36&\cdots\\
 \cdots &49&-30&19	&-11	&8	&-3	&5	&2	&7	&9	&16&25&41&\cdots\\
 \cdots&28&	-17&	11&	-6&	5&	-1&	4&	3&	7&	10&	17&	27&	44&\cdots\\
 \cdots&41&-25&16&-9&7&-2&5&3&8&11&19&30&49&\cdots\\
 \cdots&54&-33&21&-12&9&-3&6&3&9&12&21&33&54&\cdots\\
\cdots&-53&33&-20&13&-7&6&-1&5&4&9&13&22&35&\cdots\\
\cdots&-74&46&-28&18&-10&8&-2&6&4&10&14&24&38&\cdots\\
\cdots&-40&25&-15&10&-5&5&0&5&5&10&15&25&40&\cdots\\
\cdots&-61&38&-23&15&-8&7&-1&6&5&11&16&27&43&\cdots\\
\cdots&-82&51&-31&20&-11&9&-2&7&5&12&17&29&46&\cdots\\
\cdots&-48&30&-18&12&-6&6&0&6&6&12&18&30&48&\cdots\\
\cdots&-69&43&-26&17&-9&8&-1&7&6&13&19&32&51&\cdots\\
\cdots&-35&22&-13&9&-4&5&1&6&7&13&20&33&53&\cdots\\
\cdots&-56&35&-21&14&-7&7&0&7&7&14&21&35&56&\cdots\\
\cdots&-77&48&-29&19&-10&9&-1&8&7&15&22&37&59&\cdots\\
\cdots&-43&27&-16&11&-5&6&1&7&8&15&23&38&61&\cdots\\
\cdots&-64&40&-24&16&-8&8&0&8&8&16&24&40&64&\cdots\\
\cdots&-85&53&-32&21&-11&10&-1&9&8&17&25&42&67&\cdots\\
\cdots&-51&32&-19&13&-6&7&1&8&9&17&26&43&69&\cdots\\
\cdots&-72&45&-27&18&-9&9&0&9&9&18&27&45&72&\cdots\\
\cdots&-38&24&-14&10&-4&6&2&8&10&18&28&46&74&\cdots\\
\cdots&-59&37&-22&15&-7&8&1&9&10&19&29&48&77&\cdots\\
\cdots&-80&50&-30&20&-10&10&0&10&10&20&30&50&80&\cdots\\
\cdots&-46&29&-17&12&-5&7&2&9&11&20&31&51&82&\cdots\\
\cdots&-67&42&-25&17&-8&9&1&10&11&21&32&53&85&\cdots\\
\cdots&-88&55&-33&22&-11&11&0&11&11&22&33&55&88&\cdots\\

 \iddots & \vdots & \vdots & \vdots & \vdots & \vdots & \vdots & \vdots  & \vdots & \vdots & \vdots & \vdots & \vdots &\vdots &\ddots
  \CodeAfter 
\tikz \draw [line width=0.6 mm](36-|15)|-(35-|14)-|(14-|14)-|(6-|13) -| (3-|12) -| (2-|11) -| (1-|10) -- (1-|7) |- (2-|6) |- (3-|5) |- (6-|4) |- (14-|3) |-(35-|3) -|(36-|2);
\tikz \draw [line width=0.1 mm] (2-|6) -- (2-|10);
\tikz \draw [line width=0.1 mm] (3-|5) -- (3-|11);
\tikz \draw [line width=0.1 mm] (4-|5) -- (4-|12);
\tikz \draw [line width=0.1 mm] (5-|5) -- (5-|12);
\tikz \draw [line width=0.1 mm] (6-|5) -- (6-|12);
\tikz \draw [line width=0.1 mm] (10-|4) -- (10-|13);
\tikz \draw [line width=0.1 mm] (14-|4) -- (14-|13);
\tikz \draw [line width=0.1 mm] (17-|3) -- (17-|14);
\tikz \draw [line width=0.1 mm] (20-|3) -- (20-|14);
\tikz \draw [line width=0.1 mm] (23-|3) -- (23-|14);
\tikz \draw [line width=0.1 mm] (26-|3) -- (26-|14);
\tikz \draw [line width=0.1 mm] (29-|3) -- (29-|14);
\tikz \draw [line width=0.1 mm] (32-|3) -- (32-|14);
\tikz \draw [line width=0.1 mm] (35-|3) -- (35-|14);
\tikz \draw [red] [line width=0.6 mm] (1-|8) |- (2-|7);
\tikz \draw [red] [line width=0.6 mm] (2-|7) |- (4-|6);
\tikz \draw [red] [line width=0.6 mm] (4-|6) |- (6-|6);
\tikz \draw [red] [line width=0.6 mm] (6-|6) |- (9-|5);
\tikz \draw [red] [line width=0.6 mm] (9-|5) |- (22-|4);
\tikz \draw [red] [line width=0.6 mm] (22-|4) |- (36-|4);
 \end{NiceArray}}
 \]
\caption{\small{The top five blocks of the Empire State. A larger figure with
more details is given in~\cite{conwayryba}}.  The additional
red decoration is defined in the next section. Every non-zero integer appears
once to the left of the red wall. 
Every number between the red wall and the left wall 
appears with an opposite sign to the left of the left wall.}\label{table:3}}
\end{table}

\section{Ostrowski arrays and the Pell Tower}

The array $A_{m,n}$ can be defined for  
any recursion $X_{n+1}=dX_n+X_{n-1}$ and $d\geq 1$.
We call it an \emph{Ostrowski array}.
For $d=1$ the Ostrowski array is the Wythoff array.
The array extends to the left and we shall see that
it contains a building.
For $d=2$ the 
recursion produces the Pell numbers and that is
why we call the building a \emph{Pell tower}.

We limit our attention to $d>1$, since $d=1$
was fully covered in~\cite{conwayryba}.
Starting from $0,1$ the recursion $X_{n+1}=dX_n+X_{n-1}$
produces a sequence that forms the backbone of a numeration system.
If $d=1$ we get the Fibonacci numbers and Zeckendorf numeration.
If $d=2$ we
get the Pell numbers 
$(P_n)=0,1,2,5,12,29,70,169,\ldots$. 
More generally,
let $(D_n)=1,d,d^2+1,\ldots$ be the sequence for a fixed~$d$.
It is known~\cite[p. 106]{AS} that
every natural number $N$ can be represented uniquely in the form
\begin{equation}\label{eq:Os}
N=\sum_{1\leq j\leq i} d_jD_j
\end{equation}
where $D_i$ is the largest denominator less than $N$ and the digits $d_j$ satisfy
\begin{enumerate}
    \item $0\leq d_1<d$.
    \item $0\leq d_i\leq d $ for $i>1$.
    \item If $d_i=d$ then $d_{i-1}=0$.
\end{enumerate}
This numeration system is a particular case of the more general \emph{Ostrowski $\alpha$-numeration system},
which is defined from the continued fraction expansion (cfe) of~$\alpha>1$.
The denominators of the convergents in its cfe form the backbone of the numeration system.
Ostrowski numeration is particularly nice for quadratic irrationals~\cite{ostro}.
In our case we have $\alpha=\frac{d+\sqrt{d^2+4}}{2}$.
We say that the word $d_1d_2\cdots d_N$ in Equation~\eqref{eq:Os}
is an \emph{Ostrowski word}, without mentioning~$\alpha$.
Ostrowski words have letters $\{0,1,\ldots,d\}$ and
each $d$ is preceded by $0$.

We say that an Ostrowski word is \emph{trimmed} if it cannot be written as $0v$
for an Ostrowski word $v$.
The array $A_{m,n}$ for the Fibonacci
recursion has rows labelled by words in the radix order
of Equation~\eqref{eq:radix}. 
We order the Ostrowski array for the recursion $X_{n+1}=dX_n+X_{n-1}$ in the same way.
Its
rows correspond to trimmed Ostrowski words in radix order.
Any number $n$ can therefore be located from its Ostrowski representation.
This is the \emph{$d$-Ostrowski array}, 
but we shall often suppress $d$ in our notation.
The case $d=2$ is given in Table~\ref{table:4}.
\begin{table}[ht!]
\tiny{
\[
\begin{NiceArray}{cccccccccc}
0&1&2&5&12&29&70&169&408&\cdots\\
1&3&7&17&41&99&239&577&1393&\cdots\\
2&4&10&24&58&140&338&816&1970&\cdots\\
2&6&14&34&82&198&478&1154&2786&\cdots\\
3&8&19&46&111&268&647&1562&3771&\cdots\\
4&9&22&53&128&309&746&1801&4348&\cdots\\
4&11&26&63&152&367&886&2139&5164&\cdots\\
5&13&31&75&181&437&1055&2547&6149&\cdots\\
6&15&36&87&210&507&1224&2955&7134&\cdots\\
 \vdots & \vdots & \vdots & \vdots & \vdots & \vdots & \vdots & \vdots & \vdots & \ddots 
 \CodeAfter 
  \tikz \draw [line width=0.6 mm] (1-|1) -- (1-|2) |- (11-|2);
\end{NiceArray}
\]}
\caption{\small{The first ten rows of the Pell array
(the $2$-Ostrowski array), 
with an additional first column of wall terms. The table is seedless.
Within a column, there are only three differences. For instance, in
the second column the differences are $5,3,4,5,3,4,\ldots$} and in the
third column they are $12,7,10,12,7,10,\ldots$.}\label{table:4}
\end{table}

For a fixed $d$,
let $w_1,w_2,w_3,\ldots$ be the trimmed Ostrowski words in radix order, starting
from $w_1=1$.
Then $A_{m,n}$ is represented by $0^{n-1}w_m$.
The $\mathtt{out}$ operation moves from one column to the next
in the $d$-Ostrowski array.
It is defined in terms of words, but it is also possible
to give a numerical description, as in the lemma below.
Let $\beta=\frac{d-\sqrt{d^2+4}}{2}$
be the algebraic conjugate of~$\alpha$.
Note that $\alpha\beta=-1$ and $\alpha+\beta=d$ (also known as the norm and the
trace of $\alpha$).

\begin{lemma}\label{lem:1}
For every natural number $n$, $\mathtt{out}(n)=\lfloor \alpha n + \frac 1\alpha \rfloor$.
\end{lemma}
\begin{proof}
    The denominators satisfy 
    \begin{equation}\label{eq:binet}
    D_n=\frac{\alpha^n-\beta^n}{\alpha-\beta}
    \end{equation}
    from which $D_{n+1}-\alpha D_n=\beta^n$ follows. If $n=\sum_{1\leq j\leq i} d_jD_j$
    then 
    \begin{equation}\label{eq:1}
    \mathtt{out}(n)-\alpha n=\sum_{1\leq j\leq i} d_j\beta^j.    
    \end{equation}
    Even powers of $\beta$ are positive, odd
    powers are negative. The first digit is bounded by $d-1$ since $d$ needs to be preceded by $0$.
    It follows that
    \[
    d \sum_{k=1}^\infty \beta^{2k-1} - \beta < \mathtt{out}(n)-\alpha n < d \sum_{k=1}^\infty \beta^{2k},   
    \]
    which is equal to 
    \[
    \frac {d\beta}{1-\beta^2} - \beta < \mathtt{out}(n)-\alpha n < \frac {d\beta^2}{1-\beta^2}.   
    \]
    Now $\beta^2=d\beta + 1$ and $\beta = -\frac 1\alpha$. Therefore, $\mathtt{out}(n)$ is the unique
    integer in the interval $(\alpha n + \frac 1\alpha -1, \alpha n+\frac 1 \alpha)$. In other words,
    \begin{equation}\label{eq:interval}
        \mathtt{out}(n)-\alpha n\in\left(1-\frac 1\alpha,\frac 1\alpha  \right),
    \end{equation}
    and the proof is finished.
\end{proof}

An inspection of Table~\ref{table:4} shows that
if we move to the right along a fixed row, then the ratio of consecutive
numbers converges to $\alpha$. The following corollary makes this precise.

\begin{corollary}\label{cor:3}
    For a fixed $m$ and running index $n$, the differences $A_{m,n+1}-\alpha A_{m,n}$ have alternating signs and diminish in absolute value by a factor $\frac 1\alpha$.
\end{corollary}
\begin{proof}
    According to Equation~\eqref{eq:1}
    \[A_{m,n+1}-\alpha A_{m,n}=\beta^{n-1}\sum_{i=1}^kd_i\beta^k,\]
    if $w=d_1\cdots d_k$ is the $m$-th trimmed word.
    Therefore the next difference 
    $A_{m,n+2}-\alpha A_{m,n+1}$ diminishes by a factor $\beta=-1/\alpha$.
\end{proof}

We shall say that two sequences $(X_n)$ and $(Y_n)$ are
\emph{tail equivalent} if there exists an integer $j$ such
that $X_n=Y_{n+j}$ for sufficiently large $n$.
Morrison defined a table $(A_{m,n})$ 
of Fibonacci recursive sequences to be a \emph{Stolarsky array} if it
contains every natural number once, and if each Fibonacci recurrent
sequence is tail equivalent to a row in the table. 
Extending this to our recursion,
we say that the table is a $d$-Stolarsky array
if satisfies the following properties:
\begin{enumerate}
    \item Each row satisfies the recurrence $X_{n+1}=dX_n + X_{n-1}$.\label{p1}
    \item Each natural number occurs once in the table.\label{p2}
    \item For every positive recurrent sequence $(B_n)$ there exists an $m$
    such that $(A_{m,n})$ and $(B_n)$ are tail equivalent.\label{p3}
\end{enumerate}
Morrison proved that the Wythoff array is a Stolarsky array~\cite{morrison}.
His result extends to Ostrowski arrays.
\begin{theorem}\label{thm:1}
    The $d$-Ostrowski array is a $d$-Stolarsky array.
\end{theorem}
\begin{proof}
Properties \ref{p1} and \ref{p2} follow immediately from the definition of the $d$-Ostrowski array.
There exists $a>0$ and $b$ such that $X_n=a\alpha^n + b\beta^n$. For every $\epsilon>0$
there exists an $k$ such that $|X_{k+1}-\alpha X_k|<\epsilon$. In particular,
we may choose $\epsilon=-\beta$. By Property \ref{p2} there exists an $A_{m,n}$ that is
equal to $X_k$ and by Lemma~\ref{lem:1} $A_{m,n+1}=X_{k+1}$. Therefore, $A_{m,n}$ 
and $X_n$ are tail equivalent.
\end{proof}

The irrational number $\alpha>1$ and a real number $\gamma>1-\alpha$ generate the \emph{non-homogeneous Beatty sequence}
\[\mathcal B_{\alpha,\gamma}=\left\{\lfloor \alpha+\gamma\rfloor, \lfloor 2\alpha+\gamma\rfloor, \lfloor 3\alpha+\gamma\rfloor, \ldots\right\}.\]
If $\overline\alpha=\alpha/(\alpha-1)$ and $\delta$ is real then 
$\mathcal B_{\overline \alpha,\delta}$ is complementary 
to $B_{\alpha,\gamma}$, as a subset of~$\mathbb N$, if
\[
\frac{\gamma}{\alpha}+\frac{\delta}{\overline\alpha}=0,
\] provided that $\alpha>2$ and none of the $n\alpha+\gamma$ are integral,
see~\cite{fraenkel}.

\begin{corollary}\label{cor:1}
    The first column $A_{1,n}$ of the Ostrowski array is the
    non-homogeneous Beatty sequence 
    \begin{equation}\label{eq:column1}
      \left\lfloor n\cdot\frac{\alpha}{\alpha-1}-\frac 1{\alpha(\alpha-1)}\right\rfloor.  
    \end{equation}
\end{corollary}

\begin{proof}
By Lemma~\ref{lem:1} the numbers that can be written as $\mathtt{out}(n)$
form the non\--ho\-mo\-ge\-ne\-ous Beatty sequence $\lfloor n\alpha+\frac 1\alpha\rfloor$.
The first column  contains the numbers that cannot be written in this form,
i.e., the complementary Beatty sequence. It is equal to
$\mathcal B_{\overline\alpha,-{\overline\alpha}/\alpha^2}$.
\end{proof}

The first column of an Ostrowski array appears in the OEIS only for $d=1$, 
as the upper Wythoff sequence A001950. 
For $d=2$ the first column
does not occur in the OEIS, although it is very close to A081031, the positions of the white keys on a
piano keyboard, given by $\lfloor \frac {12 n - 3}7\rfloor$. The reason is that the fraction
$\frac {12}7$ is a convergent of $\frac{\alpha}{\alpha - 1}$.
The complementary Beatty sequence of the first column appears in the OEIS for
$d=1$ (lower Wythoff) and for $d=2$, sequence A082845.

We added a column $A_{m,0}$ of wall terms to the Ostrowski array.
If $w=jv$ represents $A_{m,1}$, then $v$ represents $A_{m,0}$.
It may not be an Ostrowski word.

\begin{corollary}\label{cor:2}
    For $d>1$ the sequence $A_{m,0}$ of wall terms is equal to
    $\lfloor \frac{m\alpha}{\alpha+1}\rfloor$.
\end{corollary}
\begin{proof}
A wall term is represented by the word that labels its row, with the initial digit deleted.
This is either 
an Ostrowski word (indeed, all $w$ occur since $1w$ is trimmed) or a word with prefix $d$, which is a non-Ostrowski word. 
Since all Ostrowski words occur, the wall terms contain all non-negative integers. Some integers
are repeated by the non-Ostrowski words.
These non-Ostrowski words are given by $dw$ for an Ostrowski word $w$, if we allow the empty word 
$w=\epsilon$ for the first row.
The repetitions occur at $\mathtt{out}(w)+d$, which by Lemma~\ref{lem:1} is equal to
\[
\left\lfloor (n-1)\alpha+\frac 1\alpha\right\rfloor+d=\lfloor n\alpha\rfloor. 
\]
Here we put $n-1$ to start the count at zero, to include the empty word $w$.
The number of repetitions up to but not including 
$k$ is equal to $\lfloor \frac{k}{\alpha}\rfloor$.
Therefore, at index $m=k+\lfloor \frac{k}{\alpha}\rfloor$ we have $A_{m,0}=k$.
Expressing $k$ in terms of $m$ we find that
$
k=\frac {(m+\epsilon)\alpha}{1+\alpha}
$
for some $0<\epsilon<1$. Thus, if $k$ occurs first at index $m$, then
\[
A_{m,0}=\left\lfloor \frac {m\alpha}{\alpha+1}\right\rfloor.
\]
If $k$ repeats at the next index $m+1$, then $k=n\alpha-\epsilon$ for some $n$ and $0<\epsilon<1$.
Therefore $\lfloor \frac k\alpha\rfloor=n-1=\frac k\alpha+\frac \epsilon\alpha -1$.
Now by the same argument as above, $m+1=k+\lfloor \frac{k}{\alpha}\rfloor+1=k+\frac k\alpha+\frac\epsilon\alpha$.
We get $A_{m+1,0}=\left\lfloor \frac {(m+1)\alpha}{\alpha+1}\right\rfloor.$
\end{proof}

For $d=2$ we have sequence A049472. For $d=3$ it agrees up to the thirtieth term
with A093700, which is $\lfloor n\gamma\rfloor$ for $\gamma=-\log_{10}{(3-\sqrt 8)}$.
This is because $\alpha/(\alpha+1)$ is very close to $\gamma$.

The differences $A_{m+1,1}-A_{m,1}$ between consecutive entries in the first column of the
Ostrowski array are either equal to $\lfloor \alpha\rfloor$ or
$\lceil \alpha\rceil$. 
If we code these differences by zeros and ones, then we get a Sturmian sequence.
This relation between Beatty sequences and Sturmian sequences is well studied
and there is an algorithm to convert one into the other, see~\cite{arnoux}.
Differences between terms of a non-homogeneous Beatty sequence
$\lfloor mx+y\rfloor$ follow from the rotation of the circle over $x$, 
starting from $y$.
The first column of the Ostrowski array is non-homogeneous by Equation~\eqref{eq:column1},
but the reader may check that $\frac{-1}{\alpha(\alpha-1)}$ is in the forward
orbit of zero of the rotation (it is the $(d-1)$-th iterate). 

From the second column on, the differences $A_{m+1,k}-A_{m,k}$ all seem to
follow the same pattern in Table~\ref{table:4}. Furthermore, if we apply
the $\mathtt{out}$ operation to differences in the $k$-th column, then we
seem to get the differences in the $(k+1)$-th column. This follows from 
the following additive property of the $\mathtt{out}$ operator.

\begin{corollary}\label{cor:out} 
    If $i,j,k$ are such that
    \[\mathtt{out}(i)+k=\mathtt{out}(j),\]
    then
    \[
    \mathtt{out}^2(i)+\mathtt{out}(k)=\mathtt{out}^2(j)
    \]
\end{corollary}
\begin{proof}
    By Equation~\eqref{eq:interval} 
    \[\mathtt{out}(n)-\alpha n\in\left(\frac 1\alpha-1,\frac 1\alpha\right).\]
    By Corollary~\ref{cor:3} under $\mathtt{out}$ we get
    \[\mathtt{out}^2(n)-\alpha \mathtt{out}(n)\in\left(-\frac 1{\alpha^2},\frac 1\alpha-\frac 1{\alpha^2}\right).\]
    We write $\mathtt{out}^2(p)=\alpha\mathtt{out}(p)+\epsilon_p$ for
    $p=i,j$ and $\mathtt{out}(k)=\alpha k+\epsilon_k$.
    We need to prove that
    \[
    \alpha\mathtt{out}(i)+\epsilon_i+\alpha k+\epsilon_k=\alpha\mathtt{out}(j)+\epsilon_j,
    \]
    which reduces to
    \[
    \epsilon_i+\epsilon_k=\epsilon_j.
    \]
    Since both sides of the equation are integral $\epsilon_i+\epsilon_k-\epsilon_j\in\mathbb Z$.
    From the equations above we find that
    $\epsilon_i+\epsilon_k\in\left(-1+\frac 1\alpha-\frac 1{\alpha^2},\frac 2\alpha-\frac 1{\alpha^2}\right)$
    and that $\epsilon_j\in \left(-\frac 1{\alpha^2},\frac 1\alpha-\frac 1{\alpha^2}\right)$. It follows
    that neither $\epsilon_j+1$ not $\epsilon_j-1$ are in 
    $\left(-1+\frac 1\alpha-\frac 1{\alpha^2},\frac 2\alpha-\frac 1{\alpha^2}\right)$.
    Therefore, $\epsilon_i+\epsilon_k-\epsilon_j=0.$
\end{proof}
We remark that the $\mathtt{out}$ operator is not additive on the natural numbers, but nearly: $\mathtt{out}(i+j)-\mathtt{out}(i)-\mathtt{out}(j)\in\{-1,0,1\}$. This is called the linearity defect in~\cite{carton}.
This defect is zero if $i,j,i+j$ are in $\mathtt{out}(\mathbb N)$.

The recurrence extends to negative indices under $X_{-n-1}=-dX_{-n}+X_{-n+1}$,
which produces the bi-infinite array $A_{m,n}$ for $n\in\mathbb Z$
(the ExtraFibs are now ExtraPells or ExPells). 
As in the case of ExtraFibs, 
the signs alternate and the absolute values form satisfy the recursion,
if we read from left to right.
The wall therefore reappears on the left.
It is the index from which the absolute values form a row in $A_{m,n}$
for positive $n$.
We again get a building, 
but its structure is not as regular as that of the Empire State Building.
We depict the building for $d=2$ in Table~\ref{table:5}
and we call this the Pell Tower. It is a terrace building that
displays the following patterns:
\begin{enumerate}
    \item  The distance between the walls is either $|w|$ or $|w|+1$, where $|w|$ denotes the length
    of the word that generates the row. Distance $|w|+1$ appears to be prevalent. This is illustrated by the 
    \emph{red wall} at distance $|w|$ from the right wall, where we chose red since this is the color for
    negative numbers.
    \item Columns on the left of the red wall contain positive and negative numbers. The sign depends on whether $w$ starts with $02$ or not. \label{obs2}
    \item All integers (positive and negative, but not zero) appear to the left of the red wall. 
    If a number has a negative sign left of the left wall, then it has a positive sign in between
    the red wall and the left wall, and vice versa. \label{obs3}
    \item If the left wall and the red wall coincide, then the term next to it is positive.
\end{enumerate}

\begin{table}[ht!]
\tiny{
\[ 
\begin{NiceArray}{rrrrrrrrrrl}[columns-width=auto]
&&&&&&&&&&w\\
\cdots&-70&29 &-12 &5 &-2 &1 &0 &1 &\cdots&1\\
\cdots&-41&17 &-7 &3 &-1 &1 &1 &3 &\cdots&11\\
\cdots&58&-24 &10 &-4 &2 &0 &2 &4 &\cdots&02\\
\cdots&-82&34 &-14 &6 &-2 &2 &2 &6 &\cdots&101\\
\cdots&-53&22 &-9 &4 &-1 &2 &3 &8 &\cdots&111\\
\cdots&46&-19 &8 &-3 &2 &1 &4 &9 &\cdots&021\\
\cdots&-94&39 &-16 &7 &-2 &3 &4 &11 &\cdots&102\\
\cdots&-65&27 &-11 &5 &-1 &3 &5 &13 &\cdots&1001\\
\cdots&-36&15 &-6 &3 &0 &3 &6 &15 &\cdots&1101\\
\cdots&-63&-26 &11 &-4 &3 &2 &7 &16 &\cdots&0201\\
\cdots&-77&32 &-13 &6 &-1 &4 &7 &18 &\cdots&1011\\
\cdots&-48&20 &-8 &4 &0 &4 &8 &20 &\cdots&1111\\
\cdots&51&-21 &9 &-3 &3 &3 &9 &21 &\cdots&0211\\
\cdots&-89&37 &-15 &7 &-1 &5 &9 &23 &\cdots&1021\\
\cdots&-60&25 &-10 &5 &0 &5 &10 &25 &\cdots&1002\\
\cdots&-31&13 &-5 &3 &1 &5 &11 &27 &\cdots&1102\\
\cdots&68&-28 &12 &-4 &4 &4 &12 &28 &\cdots&0202\\
\cdots&-72&30 &-12 &6 &0 &6 &12 &30 &\cdots&10001\\
\cdots&-43&18&-7&4&1&6&13&32&\cdots&11001\\
\cdots&56&-23&10&-3&4&5&14&33&\cdots&02001\\
\cdots&-84&35&-14&7&0&7&14&35&\cdots&10101\\
\cdots&-55&23&-9&5&1&7&15&37&\cdots&11101\\
\cdots&44&-18&8&-2&4&6&16&38&\cdots&02101\\
\cdots&-96&40&-16&8&0&8&16&40&\cdots&10201\\
\cdots&-67&28&-11&6&1&8&17&42&\cdots&10011\\
\cdots&-38&16&-6&4&2&8&18&44&\cdots&11011\\
\cdots&61&-25&11&-3&5&7&19&45&\cdots&02011\\
\cdots&79&33&-13&7&1&9&19&47&\cdots&10111\\
\cdots&-50&21&-8&5&2&9&20&49&\cdots&11111 \\
\cdots&49& -20&9&-2&5&8&21&50&\cdots&02111 \\
\cdots&-83& 38&-15&8&1&10&21&52&\cdots&10211 \\
\cdots&-62& 26&-10&6&2&10&22&54&\cdots&10021 \\
\cdots&-33& 14&-5&4&3&10&23&56&\cdots&11021 \\
\cdots&66& -27&12&-3&6&9&24&57&\cdots&02021 \\
\cdots&-74& 31&-12&7&2&11&24&59&\cdots&10002 \\
\cdots&-45& 19&-7&5&3&11&25&61&\cdots&11002\\
\cdots&54& -22&10&-2&6&10&26&62&\cdots&02002 \\
\cdots&-86& 36&-14&8&2&12&26&64&\cdots&10102\\
\cdots&-57& 24&-9&6&3&12&27&66&\cdots&11102 \\
\cdots&42& -17&8&-1&6&11&28&67&\cdots&02102 \\
\cdots&-98& 41&-16&9&2&13&28&69&\cdots&10202\\
\cdots&-69&29&	-11&	7&	3&	13&	29&	71&\cdots&100001\\
\cdots&-40&17&	-6&	5&	4&	13&	30&	73&\cdots&110001\\
\cdots&59&-24&	11&	-2&	7&	12&	31&	74&\cdots&020001\\
\cdots&-81&34&	-13&	8&	3&	14&	31&	76&\cdots&101001\\
\cdots&-52&22&	-8&	6&	4&	14&	32&	78&\cdots&111001\\
\cdots& 47&-19&	9&	-1&	7&	13&	33&	79&\cdots&021001\\
\cdots& -93&39&	-15&	9&	3&	15&	33&	81&\cdots&102001\\
\cdots& -64&27&	-10&	7&	4&	15&	34&	83&\cdots&100101\\
\cdots& -35&15&	-5&	5&	5&	15&	35&	85&\cdots&110101\\
\cdots& 64&-26&	12&	-2&	8&	14&	36&	86&\cdots&020101\\
\cdots& -76&32&	-12&	8&	4&	16&	36&	88&\cdots&101101\\
\cdots& -47&20&	-7&	6&	5&	16&	37&	90&\cdots&111101\\
\iddots&\vdots & \vdots & \vdots  & \vdots & \vdots & \vdots & \vdots & \vdots &\ddots
 \CodeAfter 
  \tikz \draw [line width=0.6 mm] (2-|8) -- (2-|9) |- (55-|9);
  \tikz \draw [red] [line width=0.6 mm] (3-|7) |- (5-|6) |- (9-|5) |- (19-|4) |- (43-|3) |- (55-|3);
  \tikz \draw [line width=0.6 mm] (55-|2)-|(53-|2)-|(52-|3)-|(50-|2)|-(49-|3)|-(48-|2)|-(46-|3)|-(45-|2)|-(43-|3)-|(29-|3)|-(26-|4)|-(25-|3)-|(24-|3)-|(24-|4)|-(22-|3)-|(21-|3)-|(12-|4)-|(11-|5)-|(9-|4)-|(6-|5) -| (4-|6)|-(3-|7) -| (2-|8);
   \tikz \draw [blue][line width=0.2 mm] (3-|8) -- (3-|9); 
   \tikz \draw [line width=0.1 mm] (4-|6) -- (4-|9); 
   \tikz \draw [blue][line width=0.2 mm] (5-|6) -- (5-|9);
   \tikz \draw [line width=0.1 mm] (6-|6) -- (6-|9); 
   \tikz \draw [blue][line width=0.2 mm] (11-|5) -- (11-|9); 
   \tikz \draw [blue][line width=0.2 mm] (14-|4) -- (14-|9); 
   \tikz \draw [line width=0.1 mm] (15-|4) -- (15-|9);
   \tikz \draw [blue][line width=0.2 mm] (17-|4) -- (17-|9);
   \tikz \draw [line width=0.1 mm] (19-|4) -- (19-|9);
   \tikz \draw [blue][line width=0.2 mm] (20-|4) -- (20-|9); 
   \tikz \draw [blue][line width=0.2 mm] (23-|4) -- (23-|9); 
   \tikz \draw [blue][line width=0.2 mm] (26-|4) -- (26-|9);
      \tikz \draw [line width=0.1 mm] (52-|2) -- (52-|9);
 \end{NiceArray}
 \]}
\caption{\small{The Pell Tower and its irregular left wall. On the right
we added a final column that
contains the trimmed words $w$ that generate the rows. Palindromic rows are underlined (blue if it contains zero) and they clearly
do not occur at regular distances as they do in the Empire State Building, although their
number of occurrences in blocks can be specified.
The distance between the right wall and the red wall is equal to $|w|$.}}\label{table:5}
\end{table}

We shall see that these observations can be made concrete for all $d\geq 2$
by using the
\emph{dual Ostrowski numeration system}, see~\cite[p. 181]{fogg}. 
The recursion $X_{n+1}=dX_n+X_{n-1}$ generates
the denominators $D_n$ which are the backbone of the Ostrowski system
in Equation~\eqref{eq:Os}. This is a numeration system for $\mathbb N$. 
If we extend the recursion $X_{n+1}=dX_n+X_{n-1}$ to a bi-infinite sequence, then we get the negative denominators $D_{-n}=(-1)^{n+1}D_n$.
They are the backbone of the \emph{dual} Ostrowski system, which 
is a numeration system for $\mathbb Z$. 
The following is
a special case of Proposition 6.4.19 from~\cite{fogg}.
It applies
to all $\alpha>1$, but
we only formulate it for $d>1$. The case $d=1$
was covered by Bunder who proved that the negative Fibonacci numbers
$F_{-1}, F_{-2},\ldots$ form a numeration system~\cite{bunder}.

\begin{proposition}{\label{pro:1}}
    Let $d>1$ be fixed.
    Every integer $N$ (positive or negative) can be represented uniquely in the form
    \begin{equation}\label{eq:2}
        N=\sum_{1\leq j\leq i} d_jD_{-j} 
    \end{equation} 
    with digits $d_j\in \{0,1,\ldots, d\}$ such that $d_{i+1}=0$ if $d_{i}=d$.
    The length of $w=d_1\cdots d_i$ determines the sign of $N$, which is equal to $(-1)^{|w|+1}$.
\end{proposition}

In an Ostrowski word, each $d$ is
preceded by zero. 
For the dual Ostrowski representation, each $d$ is followed by zero unless it is the final digit.
If we switch from lsd to msd representation, then the dual Ostrowski representation
is again an Ostrowski word.
Unless the msd representation starts with a~$d$. 
That is why in this case, we replace $w$ by $0w$, which represents the same number
and is an Ostrowski word. 
For the msd representation of the dual Ostrowski numeration systems, the 
initial digit is either equal to $0<j<d$, or its initial two
digits are $0d$. 
This has the pleasing effect that the words that label the rows in Table~\ref{table:5}
can also be read as msd representations in the dual system, representing terms $A_{m,n}$
left of the red wall.

\begin{lemma}
    Each integer occurs exactly once to the left of the red wall.
\end{lemma}

\begin{proof}
    The rows are labelled by Ostrowski words $w$. Each integer has a unique msd representation
    $w0^k$. To get to the red wall, we need to take $|w|$ steps to the left of the right wall.
    The number immediately to the left of the red wall has msd representation $w$ in the dual
    numeration system. If we take
    $k$ further steps, we get to $w0^k$. The length of the word determines the sign. 
    This partly explains observations \ref{obs2} and \ref{obs3}.
\end{proof}

The Ostrowski array $A_{m,n}$ with $m,n\geq 1$ starts from the right wall. 
Its counterpart, the \emph{negative Ostrowski array}, starts from the red wall.
The red wall term in the row labelled by $w$ is $A_{m,r}$ with $r=1-|w|$ (in which we
suppress that it depends on $m$) then we say that
\begin{equation}\label{negOst}
    \bar A_{m,n}=A_{m,r-n}
\end{equation}
is the negative Ostrowski array for $m,n\geq 1$.
Inhabitants of the Pell Tower enjoy the view of these two gardens.
The number of terms inside the building on level $m$ is equal to $1-r=|w|$, where $w$
is the $m$-th Ostrowski word in the radix order.

The operation $n\to\mathtt{out}(n)$ moves one step to the right in the Ostrowski array.
Its counterpart $n\to\text{nut}(n)$, the negative $\mathtt{out}$, takes one step to the left in
the negative Ostrowski array. If $u$ is the msd dual representation
of $n$, then $u0$ is the msd dual representation of $\text{nut}(n)$. It appends a zero.
The observations on the Pell tower that we made above are all consequences of the
following lemma.

\begin{lemma}\label{lem:3}
For any integer $n$ we have \emph{$\mathtt{nut}$}$(n)=\lceil -n \alpha \rceil$.
\end{lemma}
\begin{proof}
    The Equation~\eqref{eq:binet} holds for all integers and
    as in the proof of Lemma~\ref{lem:1} we find 
    $D_{-n-1}+\alpha D_{-n}=\alpha^{-n}.$
    If $n=\sum_{1\leq j\leq i} d_jD_{-j}$
    then
    \begin{equation}\label{eq:nutalpha}
      \mathtt{nut}(n)+\alpha n=\sum_{1\leq j\leq i} d_j\alpha^{-j}\geq \frac 1\alpha.  
    \end{equation} 
    The sum $\sum_{1\leq j\leq i} d_j\alpha^{-j}$ is maximized by
    taking all digits equal to $d$ at odd indices and zero at even indices.
    Since $d(\alpha^{-1}+\alpha^{-3}+\cdots)=1$ we get
    $\mathtt{nut}(n)+\alpha n<1$.
\end{proof}

This expression for $\mathtt{nut}$ is simpler than the one for $\mathtt{out}$ in Lemma~\ref{lem:1}.
It implies that the lsd dual representation can be determined by
a simple divide and round.
Indeed, to find the lsd representation of $N_1$, compute $N_2=\lceil -N_1/\alpha\rceil$ and put
$x_1=N_1-\lceil -N_2\alpha\rceil$. 
The lemma implies that $N_1-x_1=\text{nut}(N_2)$,
hence the representation of $N_1$ has least significant digit $x_1$. 
Continue with $N_2$ to find its digit $x_2$, etc.
Terminate as soon as $0< N_k \leq d$ and put $x_k=N_k$. 
This is a standard digit generating
procedure
known as the greedy beta expansion~\cite{dajani}. To see that it produces Ostrowski words,
observe that digit $d$ occurs only if $N_1\in [-N_2\alpha+d, -N_2\alpha+\alpha)$
which has length $\alpha-d=\frac 1\alpha$. In particular $-N_2\alpha=N_1-d-\epsilon$
for some $\epsilon<\frac 1\alpha$. Therefore 
\[
\frac {-N_2}{\alpha}=-N_2\alpha+dN_2=N_1-d-\epsilon + dN_2,
\]
which rounds up to $N_3=N_1-d+dN_2$ with digit
\[
x_2=N_2-\lceil -N_3\alpha\rceil<\epsilon \alpha
\]
If $x_1=d$ then $x_2=0$. The greedy beta expansion produces Ostrowski words.

\begin{corollary}\label{cor:3b}
    For a fixed $m$ and running index $n$, the sums $\bar A_{m,n+1}+\alpha \bar A_{m,n}$ 
    are positive and diminish by a factor $\frac 1\alpha$.
    Furthermore $\bar A_{m,1}+\alpha \bar A_{m,0}>1$. Therefore, 
    the largest index $n$ in the $m$-th row such that
    $A_{m,n-1}+\alpha A_{m,n}<1$ is at $r-1$. This is the initial term of the
    negative Ostrowski array.
\end{corollary}
\begin{proof}
    Let $w=d_1\cdots d_i$ be the word that represents the $m$-th row. 
    Then the msd representation of $\bar A_{m,n}=A_{m,r-n}$ is $w0^{n-1}$.
    According to Equation~\eqref{eq:nutalpha}
    \[A_{m,-n-1+r}+\alpha A_{m,-n+r}=\frac 1{\alpha^{n-1}}\sum_{1\leq j\leq i} d_j\alpha^{-j}.\]
    Thus the next sum 
    $A_{m,n+2}+\alpha A_{m,n+1}$ diminishes by a factor $1/\alpha$ and
    $A_{m,r-1}+\alpha A_{m,r}\in (0,1)$. 
    Equation~\eqref{eq:binet} holds for all $n$ and therefore $A_{m,n+1}+\alpha A_{m,n}$ increases
    by a factor $\alpha$ if $n$ increases by one, for the entire row.
    The index $r-1$, the first column of the negative Ostrowski array, is the unique
    index such that the sum is
    in $[\frac 1\alpha, 1)$.
\end{proof}

\begin{corollary}\label{cor:4}
    For every row $A_{m,n}$ in the Ostrowski array (fixed $m$) there exists a 
    row $\bar A_{k,n}$ in the negative Ostrowski array (fixed $k$) and a number $i\in\{0,1\}$ 
    such that $A_{m,n}=|\bar A_{k,i+n}|$.
\end{corollary}

If $i=0$ then the left wall and the right wall coincide. If $i=1$ then 
there is a space of one, a terrace, between the red wall and the left wall.

\begin{proof}
Let $N=A_{m,1}$ be the first term of the row. 
Both $-N$ and $N$ occur somewhere in the negative Ostrowski array defined
in Equation~\eqref{negOst}.
The next terms are, respectively, $\text{nut}(-N)=\lceil N\alpha\rceil$ and 
$\text{nut}(N)=\lceil -N\alpha \rceil$. 
In absolute value these terms are $\lceil N\alpha\rceil$ and 
$\lfloor N\alpha \rfloor$. One and only one of these two absolute values is equal to $A_{m,2}$.
Two recursive sequences are equal if they have two identical consecutive terms.
Therefore $A_{m,n}$ occurs (possibly as a tail) in a unique row of
the negative Ostrowski array.

Now we know that the row $A_{m,n}$ occurs in a row $\bar A_{k,n}$ of
the negative array, we want to locate where it starts. Signs in the
negative array are alternating, and therefore $\bar A_{k,2}$ and
$\bar A_{k,3}$ have opposite signs. It follows that 
\[
-\frac 1\alpha|<|\bar A_{k,2}|-\alpha|\bar A_{k,3}|<\frac 1\alpha.
\]
By Lemma~\ref{lem:1}
it follows that $\mathtt{out}\left(|\bar A_{k,2}|\right)=|\bar A_{k,3}|.$
Therefore the row $|\bar A_{k,n}|$ running from index $n=2$ onward occurs
as a (tail of a) row in the Ostrowski array, which must be row $m$ by uniqueness.
We conclude that $i=0$ or $i=1$.
\end{proof}

All integers occur left of the red wall. The left wall marks where an alternating
copy of the Ostrowski array starts. From each pair $\{-n,n\}$ it contains one.
The other occurs on the terraces, between the left wall and the red wall.

\begin{corollary}\label{cor:5}
    If the red wall and the left wall coincide, then the number left of it is positive.
    Indeed, a natural number $N$ is next to these two coinciding walls if and only if
    $\alpha N-\lfloor \alpha N\rfloor \in \left[\frac 1\alpha, 1-\frac 1\alpha\right]$.
\end{corollary}
\begin{proof}
All non-zero integers appear once left of the red wall. Half of the integers, one from each pair $\{-N,N\}$,
appears left of the left wall. The other half is on the terrace, the space between the left wall and the
red wall. If the left wall and the red wall coincide, then the term next to it is one from a pair $\{-N,N\}$.
In other words, the walls coincide if and only if both $N$ and $-N$ are in
the first column of the negative Ostrowski array.

A number $N$ is in the first column of the negative Ostrowski array if and only if
$\mathtt{out}(N)+\alpha N=\lceil -\alpha N\rceil +\alpha N\in \left[\frac 1\alpha,1\right)$. 
Both numbers $-N$ and $N$ are in the first column if and only if
\[
\alpha N-\lfloor \alpha N\rfloor \in \left[\frac 1\alpha, 1-\frac 1\alpha\right].
\]
It follows from the unique ergodicity of the rotation
that the fraction of numbers with
this property is equal to the length of the interval $\left[\frac 1\alpha, 1-\frac 1\alpha\right]$,
which is approximately $0.172..$ if $d=2$. This is why most numbers in the first column
are on the terrace in Table~\ref{table:5}.   

Which of the two $\{-N,N\}$ is on the terrace? Consider the positive number $N$. Its neighbor
$\mathtt{nut}(N)=\lceil -\alpha N\rceil$ has absolute value $\lfloor \alpha N\rfloor\geq \alpha N-\frac 1\alpha$.
Therefore 
\[\mathtt{out}(N)=\left\lfloor \alpha N+\frac 1\alpha\right\rfloor=\lfloor \alpha N\rfloor =|\mathtt{nut}(N)|.\]
The number $N$ has the same neighbor (in absolute value) in both the Ostrowski array and the negative Ostrowski array.
The left wall and the red wall coincide at $N$.
\end{proof}

The numbers that are not on the terrace in Table~\ref{table:5} are sequence A276879 in the OEIS. 
This sequence has density $0.172..$ as we have seen but
for larger $d$, the density of the natural numbers that are not on the terrace increases to one.

We extend the notion of a Stolarsky array to include recursive sequences that contain negative numbers:
\begin{enumerate}
    \item Each row satisfies the recurrence $X_{n+1}=-dX_n + X_{n-1}$.
    \item Each non-zero integer occurs once in the table.
    \item For every recurrent sequence $(B_n)$ there exists an $m$
    such that $(A_{m,n})$ is tail equivalent to $(B_n)$ or to $(-B_n)$.
\end{enumerate}

Our previous results imply:

\begin{theorem}\label{thm:3}
    The negative Ostrowski array  is a Stolarsky array.
\end{theorem}

The sequence of denominators $(D_n)$ is palindromic and so is the sequence 
$(E_n)$ given by $\cdots,d^2+2,-d,2,d,d^2+2,\cdots$. 
These are the so-called \emph{companion numbers}~\cite{bicknell}
which satisfy
\begin{equation}\label{eq:companion}
    E_n=\alpha^n+\beta*n.
\end{equation}
It is not hard to prove that
all palindromic sequences are multiples of $(D_n)$ or $(E_n)$, if
we allow multiples of the companion numbers to be halves if $d$ is even.
Following the Fifis and Lulus from~\cite{conwayryba},
let's call the multiples of $(D_n)$ Deedees and call the multiples of $(E_n)$
Edees. 

The Empire State Building is divided in blocks, counting from zero, 
where block $k$ consists
of all rows that are labelled by words of length 
$|w|=2k$ or $|w|=2k+1$. 
Fifis occur in the even blocks and Lulus occur in the odd blocks.
We modify this definition for
the Pell Tower and define \emph{block} $k$ to contain the
rows labelled by
words of length $|w|=2k-1$ or $|w|=2k$.
The initial word of block $k$ is $10^{2k-3}1$ (or $1$ if $k=1$)
and the final word is $(0d)^k$. Both rows are palindromes.
The distribution of Deedees and Edees over these blocks is not
as nice as for the Empire State Building, but we can still count
how many there are per block.

\begin{theorem}
The number of Deedees in block $k$ is equal to the number of times $k$
occurs in the sequence $\lfloor\log_{\alpha}(n)\rfloor + 1$. 
The number of
Edees in block $k$ is equal to the number of times $k$ occurs in
$\lfloor\log_{\alpha}(n(\alpha-\beta))\rfloor + 1$ where the
$n$ are halves if $d$ is even.
\end{theorem}
\begin{proof}
Consider $(jD_n)$ for some fixed $j$.
The first term of this sequence in the negative Ostrowski array
occurs at index $-i$ such that
$jD_{-i-1}+\alpha jD_{-i}<1\leq jD_{-i}+\alpha jD_{-i+1}$. 
We have $jD_{-i-1}+\alpha jD_{-i}=\frac j{\alpha^i}$
and so $i-1\leq\log_{\alpha}(j)<i$, or equivalently
$i=\lfloor\log_{\alpha}(j)\rfloor + 1$.
We determined the entry that is in the first column
of the negative array.
What is the entry in the first column of the Ostrowski array?
If the red wall and the left wall coincide, it is $jD_i$ and
if not then it is $jD_{i+1}$. The number of entries inside the
Pell tower is either $2i-1$ or $2i$. Hence $(jD_n)$ is in block $\lfloor \log_{\alpha}(j)\rfloor + 1$.

The computation for the Edees is identical.
The first term of $(jE_n)$ (the $j$ are halves if $d$ is even)
in the negative Ostrowski array
occurs at index $-i$ such that
\[
jE_{-i-1}+\alpha jE_{-i}<1\leq jE_{-i}+\alpha jE_{-i+1}.
\]
We have $jE_{-i-1}+\alpha jE_{-i}=\frac {j(\alpha-\beta)}{\alpha^i}$
and so $i=\lfloor\log_{\alpha}(j(1+\alpha))\rfloor + 1$.
We conclude that $(jE_n)$ is in block $\lfloor \log_{\alpha}(j)\rfloor + 1$.
\end{proof}

So where do we find $100D_n$ in Table~\ref{table:5}? It is in block $6=\lfloor \log_{\alpha}(100)\rfloor + 1$.
The first entry in the negative Ostrowski array is $100D_{-6}=-7000$. 
The first entry in the Ostrowski array is $16900$ which in Pell numeration is
given by $110101110101$.
This word represents $-7000$ in msd dual Pell numeration.
It is possible to compute the location of the palindromes in the table, but 
there does not seem to be a nice formula for these locations.
Conway and Ryba were able to find nice formulas for the palindromes
of the Empire State Building using Fibonometry, which we will consider
in the next section.

We conclude this section with some remarks on the case
of Fibonacci numbers. 
We do not supply proofs, as they are either very similar to
the proofs above or they are consequences of the results
of Conway and Ryba. 
Bunder proved that each integer can be written as $N=\sum_{1\leq j\leq i} d_jF_{-j}$
for digits $d_j\in\{0,1\}$ such that $d_{j+1}=0$ if $d_j=1$.
Bunder's algorithm Z to determine the expansion is not very complicated, but it 
involves a few different operations and the proof of its correctness requires a bit of work. 
There is a simpler algorithm!
The analogue of our Lemma~\ref{lem:3} holds for the negative base $-\gamma$,
where $\gamma=\frac{1+\sqrt 5}2$ is the golden ratio.
To determine the negative Zeckendorf representation, divide and round by $-\gamma$ and terminate at one.
We can also put a red wall 
within the Empire State Building. Since the array involves a seed, which is absent
for Ostrowski arrays for $d>1$, the distance between the red wall and the left wall
is either 1 or 2. The $n$-th block (counting blocks from zero)
is divided into two parts, starting with $F_{2n}$ rows
of distance 2 and ending with $F_{2n+1}$ rows of distance 1. 
The negative Wythoff array is a Stolarsky array.

\section{Ostronometry}

The Fibonacci numbers and the Lucas numbers satisfy many interesting 
identities. The  oldest seems to
be Cassini's identity from 1680, if not earlier: 
\[
F_{n+1}F_{n-1}-F_n^2=(-1)^n.
\]
This was only the start of an ever growing list.
Bicknell~\cite{bicknell} observed that these identities all
extend to our denominators $D_n$, and we retrace her footsteps. For instance, Cassini's identity extends to
\[
D_{n+1}D_{n-1}-D_n^2=(-1)^n.
\]
In these generalized identities the companion numbers $E_n$
take the role of the Lucas numbers. 

This conversion depends on a trick involving trigonometry,
which is why Conway and Ryba call it Fibonometry~\cite{fibo}.
The trick is apparently due to Vajda~\cite{Vajda} and extends to $D_n$ and $E_n$
and it seems natural to call this Ostronometry.
By definition
\[
\sin(nz)=\frac{e^{inz}-e^{-inz}}{2i}\ ,\ \ \ \cos(nz)=\frac{e^{inz}+e^{-inz}}2. 
\]
The recursive sequence
that starts with $2$ and $d$ produces the companion
numbers $(E_n)$. We write $\Delta=d^2+4$. Vajda's trick
puts $z=\frac {\pi} 2 - i\log(\alpha)$.
By Equations~\eqref{eq:binet} and~\eqref{eq:companion} we find
\begin{equation}\label{eq:trick}
\sin(nz) = \frac {D_n \sqrt{\Delta}}{2}\cdot i^{n-1} \ ,\ \ \ \cos(nz)=\frac{E_n}{2}\cdot i^n.  
\end{equation}
If $d=1$ then $D_n$ and $E_n$ are the Fibonacci and Lucas numbers. 
Now trigonometric equations transform into Cassini-like identities.
Fibonacci identities transform into Ostrowski identities under
$F_n\to D_n\cdot\sqrt{d^2+4}/\sqrt{5}$ and $L_n\to E_n$.
For instance, the fundamental equation $\cos^2(nz)+\sin^2(nz)=1$ becomes
\[
E_n^2-\Delta D_n^2=(-1)^n 4,
\]
and we see that $(D_n,E_n)$ solves a Pell equation. Cassini's identity
follows from
\begin{equation*}
\sin^2(x) -  \sin^2(y) = \sin(x + y) \sin (x - y),
\end{equation*}
if $x=nz$ and $y=z$.
The trigonometric Jacobi identity from~\cite{conwayryba}
is equal to
\[
\sin(a)\sin(b-c)+\sin(b)\sin(c-a)+\sin(c)\sin(a-b)=0.
\]
It should be read in terms of $az, bz, cz$ and transforms to
\[
(-1)^cD_aD_{b-c}+(-1)^aD_bD_{c-a}+(-1)^bD_cD_{a-b}=0.
\]
Cassini's identity is the special case in which $a,b,c$ is equal to $n+1,n,1$.
A generalization of Cassini's identity, which is sometimes named after d'Octagne, is
\[
F_mF_{n+1}-F_{m+1}F_n=(-1)^nF_{m-n}.
\]
It is another consequence of the trigonometric Jacobi equation. By Ostronometry we get
\begin{equation}\label{eq:3}
D_mD_{n+1}-D_{m+1}D_n=(-1)^nD_{m-n}.
\end{equation}

Ostronometry can also be used to demonstrate divisibility properties of
the denominators. It is well known that $F_d$ divides $F_n$ if
$d$ divides $n$. By Fibonometry, this follows from the
fact that by De Moivre $\sin(nz)$ is a sum of $\sin^j(z)\cos^{n-j}(z)$ for odd $j$. 
More can be said. If we take $m=n+1$ in Equation~\eqref{eq:3} then we find that
two consecutive denominators
are relatively prime. Therefore $\gcd(D_m,D_n)=\gcd(D_{m-n},D_n)$ and by a run
of the lazy Euclidean algorithm, which subtracts one $n$ at a time instead of
a multiple, we conclude that
$\gcd(D_m,D_n)=D_{\gcd(m,n)}$. In particular, $D_n$ divides $D_m$ if and only
if $n$ divides $m$, which is a well known fact for Fibonacci numbers.
Another fun fact is Carmichael's Theorem, which says that the product of any $n$ consecutive Fibonacci
numbers is divisible by $F_1\cdots F_n$. It is a consequence of d'Octagne's
identity~\cite[p 74]{Vajda} and hence extends to $D_n$.

Vajda's trick replaces $e^{iz}$ by a fundamental
solution of the Pell equation $|X^2-\Delta Y^2|=1$.
\[
\cos(z)+i \sin(z)\longleftrightarrow \frac {E_1 + D_1\sqrt {\Delta}}{2}\cdot i.
\]
The reason why this works is its stability under $n$-th powers
\[
\cos(nz)+i \sin(nz)\longleftrightarrow \frac {E_n + D_n\sqrt {\Delta}}{2}\cdot i^n.
\]
The left-hand side follows from De Moivre's identity. The right-hand side
follows from the fact that solutions of
the Pell equation form a cyclic group. 
Vajda uses the hyperbolic sine and cosine, which give cleaner formulas,
but their identities are less familiar.

It is possible to adapt Vajda's trick to other rows in the Ostrowski table, although
it gets a little cumbersome. Pick a row $m$ in the Ostrowski table
and write $Y_n=A_{m,n}$. The generalized Binet formula gives
\[
Y_n = \left(a+\frac{b}{\sqrt{\Delta}}\right)\alpha^n\ +\ \left(a-\frac{b}{\sqrt{\Delta}}\right)\beta^n  
\]
for $2a=Y_0$ and $ad+b=Y_1$. The numbers $Y_n$ need companions $X_n$ to solve a Pell
equation $X_n^2 - \Delta Y_n^2 = (-1)^{n}C$ for some constant $C$. Asymptotically, $X_n$ needs to be
equal to $\sqrt{\Delta}Y_n$ and so
\[
X_n = \left(a\sqrt{\Delta}+b\right)\alpha^n\ +\ \left(-a\sqrt{\Delta}+b\right)\beta^n.  
\]
We have that
\begin{equation}\label{eq:norm}
X_n+Y_n\sqrt{\Delta}=(X_0+Y_0\sqrt\Delta)\cdot\frac{E_n+D_n\sqrt{\Delta}}2.
\end{equation}
By taking the norm it follows that the pairs $(X_n,Y_n)$ solve the Pell equation above with $C=X_0^2-Y_0^2\Delta$.
Choose $\phi$ such that
\[
\sin(\phi) = \frac { Y_0 \sqrt{\Delta}}{i\sqrt{C}} \ ,\ \ \ \cos(\phi)=\frac{X_0}{\sqrt{C}}. 
\]
Vajda's trick in Equation~\eqref{eq:trick} combined with Equation~\eqref{eq:norm} gives
\[
\frac{X_n+Y_n\sqrt{\Delta}}{\sqrt C}=(\cos(\phi)+i\sin(\phi))(\cos(nz)+i\sin(nz))i^{-n}.
\]
According to Equation~\eqref{eq:norm} we have $Y_n=X_0D_n+Y_0E_n$, which corresponds to
the computation of the imaginary part of $(\cos(\phi)+i\sin(\phi))(\cos(nz)+i\sin(nz))$.
From this we find $Y_n$ and hence $X_n$.
Vajda's trick for Ostrowski rows is
\begin{equation}\label{eq:modi}
\sin(nz+\phi) = \frac {Y_n \sqrt{\Delta}}{\sqrt{C}}\cdot i^{n-1} \ ,\ \ \ \cos(nz+\phi)=\frac{X_n}{\sqrt{C}}\cdot i^n.  
\end{equation}
It is a little more cumbersome since it has an extra angle $\phi$ and a constant $C$. For instance, 
if we read $a$ as $az+\phi$, $b$ as $bz+\phi$ and $c$ as $z$, then the trigonometric Jacobi identity transforms to
\[
Y_bY_{a-1}-Y_aY_{b-1}=(-1)^bD_{a-b}\cdot \frac C4.
\]

\section{Concluding remarks}

Negative numbers were never considered to be satanic, this is a modern 
myth~\cite{nothaft}, but it is fair to say that they have not received as much attention as positive numbers.
We should treat all numbers, negative and positive, odd and even, with equal respect, regardless
of orientation or parity.
Our understanding of the combinatorial and dynamical properties of
negative beta-expansions has progressed tremendously
thanks to works of Charlier, Frougny, Ito, Pelantov\'a, Steiner, and many others. 
An overview of the literature with open problems is given in \cite[Ch 2]{steiner}.
Negative bases have recently been implemented in
the automatic theorem prover {\tt Walnut}, which already devoured its first
conjectures~\cite{walnut}. 
Labb{\'e} and Lep{\v{s}}ov{\'a} 
recently found an interesting new type of negative Zeckendorf numeration 
from Wang tiles~\cite{labbe}.

We restricted our attention to the recursion $X_{n+1}=dX_n+X_{n-1}$. How about other
recursions? A natural choice is the Tribonacci recursion
$X_{n+1}=X_{n}+X_{n-1}+X_{n-2}$ that was considered in~\cite{tribo}.
It turns
out that it is very difficult to find Conway-Ryba type of results for this recursion.
It would be nice if there is some sort of Tribonometry, but it may not exist.
However, it is possible to define the bi-infinite Ostrowski array for arbitrary $\alpha>1$
by using the dual Ostrowski numeration system. 
I am grateful to one of the referees for
pointing that out. 
Is its first column again a non-homogeneous Beatty sequence for all $\alpha>1$?
What can be said about the building inside a general Ostrowski array?

Overt nationalism has regained respectability once again, so
let me highlight the abundance of Dutch mathematicians  in this
paper. Maarten Bunder, Gerrit Lekkerkerker
(who preceded Edouard Zeckendorf), John Pell, and Willem Wythoff all studied or 
worked at the
University of Amsterdam, just like me.
The Empire State Building is located in the former New Amsterdam in what
was then the colonial province of New Netherland.
Edouard Zeckendorf grew up near Li\`ege, where he studied medicine, but his parents were 
from Amsterdam.
This city is infamous for various reasons, yet its uncanny
connection with recursion has so far gone unnoticed.

\subsection*{Acknowledgement}
I am grateful to the two referees who both made some very useful suggestions and expert remarks on earlier versions of this paper.


\EditInfo{September 12, 2023}{August 7, 2024}{Emilie Charlier, Julien Leroy and Michel Rigo}

\end{document}